\documentclass[a4paper, 12pt, reqno, dvipsnames]{amsart}

\usepackage[normalem]{ulem} 
\usepackage{lmodern}
\usepackage[english]{babel}
\usepackage[latin1]{inputenc}
\usepackage{microtype}
\usepackage{amsfonts,amssymb,amsmath,amsthm,mathrsfs,stmaryrd}
\usepackage{mathtools}
\usepackage{color, graphicx}
\usepackage[all]{xy}
\usepackage{hyperref} 
\usepackage[inline]{enumitem}
\usepackage{cleveref}
\usepackage{a4wide}
\usepackage{upgreek}
\usepackage{etoolbox}

\makeatletter
\patchcmd{\@setaddresses}{\indent}{\noindent}{}{}
\patchcmd{\@setaddresses}{\indent}{\noindent}{}{}
\patchcmd{\@setaddresses}{\indent}{\noindent}{}{}
\patchcmd{\@setaddresses}{\indent}{\noindent}{}{}
\makeatother

\makeatletter
\DeclareMathSizes{\@xpt}{\@xpt}{6}{5}
\makeatother

\makeatletter
\def\namedlabel#1#2{\begingroup
    #2%
    \def\@currentlabel{#2}%
    \phantomsection\label{#1}\endgroup
}
\makeatother

\allowdisplaybreaks

\theoremstyle{plain}

\newtheorem{theorem}{Theorem}[section]
\newtheorem{lemma}[theorem]{Lemma}
\newtheorem{proposition}[theorem]{Proposition}
\newtheorem{corollary}[theorem]{Corollary}
\newtheorem*{theorem*}{Theorem}

\theoremstyle{definition}
\newtheorem{definition}[theorem]{Definition}
\newtheorem{example}[theorem]{Example}

\theoremstyle{remark}
\newtheorem{remark}[theorem]{Remark}

\newtheorem{remarks}[theorem]{Remarks}

\newcommand{\N}{\mathbb{N}}
\newcommand{\Z}{\mathbb{Z}}

\newcommand{\R}{\mathbb{R}}

\newcommand{\K}{\Bbbk}

\newcommand{\op}[1]{#1^{\mathrm{op}}} 
\newcommand{\mf}[1]{\mathfrak{#1}} 
\newcommand{\ms}[1]{\mathsf{#1}}

\newcommand{\gi}[1]{{#1}^{-1}} 

\newcommand{\id}{\mathsf{id}} 

\newcommand{\alg}{\mathsf{Alg}}
\newcommand{\algk}{\alg_{\K}} 
\newcommand{\Calgk}{\mathsf{CAlg}_{\K}} 
\newcommand{\vectk}{\mathsf{Vect}_{\K}} 
\newcommand{\Set}{{\mathsf{Set}}} 
\newcommand{\Top}{{\mathsf{Top}}} 
\newcommand{\Rmod}[1]{\M_{#1}} 
\newcommand{\gparcom}[1]{\mathsf{gPCom}^{#1}} 
\newcommand{\com}[1]{\mathsf{Com}^{#1}} 
\newcommand{\Cov}{\mathsf{Cov}} 
\newcommand{\Com}{\mathsf{Com}}
\newcommand{\Mod}{\mathsf{Mod}}

\newcommand{\Ind}{\mathsf{Ind}}

\newcommand{\cC}{{\mathcal C}}

\newcommand{\cG}{{\mathcal G}}

\newcommand{\cI}{{\mathcal I}}
\newcommand{\cJ}{{\mathcal J}}

\newcommand{\cO}{{\mathcal O}}

\newcommand{\cZ}{{\mathcal Z}}

\newcommand{\M}{\mathsf{Mod}} 
\newcommand{\I}{\mathbb{I}} 
\newcommand{\calpha}{\mf{a}}
\newcommand{\clambda}{\mf{l}}
\newcommand{\crho}{\mf{r}}

\newcommand{\ie}{i.e.~}
\newcommand{\eg}{e.g.~}

\renewcommand{\top}{\mathbf{\uptau}}

\newcommand{\ot}{\otimes}
\newcommand{\bul}{\bullet}

\newbox\pullbackbox
\setbox\pullbackbox=\hbox{\xy 0;<1mm,0mm>: \POS(2,0)\ar@{-} (0,2) \ar@{-} (4,2)
\endxy}

\newbox\pushoutbox
\setbox\pushoutbox=\hbox{\xy 0;<1mm,0mm>: \POS(2,2)\ar@{-} (0,0) \ar@{-} (4,0)
\endxy}
\def\pushout{\copy\pushoutbox}

\definecolor{bostonuniversityred}{rgb}{0.8, 0.0, 0.0}

\title[Globalization for geometric partial comodules]{Globalization for geometric partial comodules}

\author{Paolo Saracco}
\address{D\'epartement de Math\'ematique, Universit\'e Libre de Bruxelles, Boulevard du Triomphe, B-1050 Brussels, Belgium.}
\urladdr{\url{sites.google.com/view/paolo-saracco}}
\urladdr{\url{paolo.saracco.web.ulb.be}}
\email{paolo.saracco@ulb.be}

\author{Joost Vercruysse}
\address{D\'epartement de Math\'ematique, Universit\'e Libre de Bruxelles, Boulevard du Triomphe, B-1050 Brussels, Belgium.}
\urladdr{\url{joost.vercruysse.web.ulb.be}}
\email{joost.vercruysse@ulb.be}

\thanks{PS is a Charg\'e de Recherches of the Fonds de la Recherche Scientifique - FNRS and a member of the National Group for Algebraic and Geometric Structures and their Applications (GNSAGA-INdAM).
JV thanks the FNRS (National Research Fund of the French speaking community in Belgium) for support via the MIS project `Antipode' (Grant F.4502.18) and the FWB (F\'ed\'eration Wallonie-Bruxelles) for support via the ARC project "From algebra to combinatorics, and back". \hfill \\
\hfill \\
This version of the article has been accepted for publication, after peer review. The final publication is available at Elsevier via \href{https://doi.org/10.1016/j.jalgebra.2022.03.013}{doi.org/10.1016/j.jalgebra.2022.03.013}.}
\keywords{Globalization; monoidal category; coalgebra; geometric partial comodule; partial action; partial comodule algebra.}
\subjclass[2010]{16T15, 16W22, 18A40} 

\usepackage{fancyhdr}

\begin{document}

\begin{abstract}
We discuss globalization for geometric partial comodules in a monoidal category with pushouts and we provide a concrete procedure to construct it, whenever it exists. The mild assumptions required by our approach make it possible to apply it in a number of contexts of interests, recovering and extending numerous {\em ad hoc} globalization constructions from the literature in some cases and providing obstruction for globalization in some other cases.
\end{abstract}

\maketitle

\fancyhf{}
\renewcommand{\headrulewidth}{0pt}
\thispagestyle{fancy}
\cfoot{\smallskip\footnotesize $\begin{gathered}\includegraphics[scale=0.5]{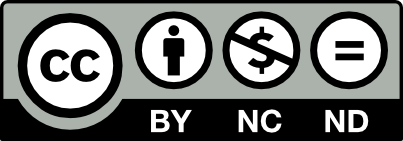}\end{gathered}$ \ \copyright\, 2022. This manuscript version is made available under the \href{https://creativecommons.org/licenses/by-nc-nd/4.0/}{CC-BY-NC-ND 4.0} license.}



\section{Introduction}

The notion of partial action of a group on a set (also known as partial dynamical system) appeared in \cite{Exel1} within the theory of operator algebras as an approach to $C^*$-algebras generated by partial isometries, permitting, in particular, the study of their $K$-theory, ideal structure and representations. 
The point of view of crossed products by partial actions of groups has been enormously successful for classifying $C^*$-algebras 
and in the last few years the investigation of topological and $C^*$-algebraic partial dynamical systems experienced a period of intense activity (for instance, recently, partial coactions of $C^*$-bialgebras and of $C^*$-quantum groups on $C^*$-algebras were introduced and studied in \cite{Kraken}).

At the same time, the study of partial actions and representations from a more algebraic point of view attracted the attention of numerous researchers in the field and soon it became an independent topic of interest in algebra and ring theory, resulting in remarkable applications and theoretic development (for an idea of the impact of partial (co)actions on contemporary Mathematics, we refer the reader to the recent survey \cite{Doku2} and the references therein). In particular, motivated by an extension of classical Galois theory \cite{DokuFerreroPaques}, partial actions entered the realm of Hopf algebras \cite{CaenepeelDeGroot}, \cite{Caenepeel-Janssen}.

One of the relevant questions in the study of partial actions is the problem of the existence and uniqueness of a globalization (also called an enveloping action). 
Any action of a group on a set induces a partial action of the group on any subset by restriction (see Example \ref{ex:induced} below). The other way around, ``globalizing'' a given partial action means to find a (minimal) global action such that the initial partial action can be realized as the restriction of this global one.
The aim of the restriction and globalization procedures is to relate partial and global actions in such a way 
that results can be extended from the global to the more general partial setting and, conversely, general results in the partial case can be used to refine and complete our understanding of global actions. 
In addition, globalizable partial actions play a key role in the development of Galois theory of partial group actions in \cite{DokuFerreroPaques}. 

The study of this problem begun in the context of partial actions of groups on topological spaces in \cite{Abadie} and, independently, \cite{KellendorkLawason}, where it was proved that, up to isomorphism, each partial action can be globalized (see also \cite{Novikov}). 
For the partial actions of a group on a unital associative algebra, 
a criterion for the globalizability was given in \cite[Theorem 4.5]{DokuExel}. This criterion was generalized to the so-called left $s$-unital rings  (i.e.\ rings with left local units) in \cite{DokuDelRioSimon} and it was also used to analyse when a partial action on a semiprime ring is globalizable \cite{BemmFerrero,CortesFerrero}.
In \cite{DokuExelSimon}, a globalization for twisted partial actions was established and in \cite{Ferraro} the problem of globalizability of partial actions on non-necessarily unital rings, algebras and $C^*$-algebras was studied. 
In the theory of partial (co)actions of Hopf algebras, one of the first results obtained was exactly that every partial action of a Hopf algebra on a unital algebra admits a suitable globalization \cite{AlvesBatista2,AlvesBatista}, which however is not necessarily unital. Similar theorems were proved in other contexts such as partial actions of Hopf algebras on $\K$-linear categories \cite{AlvaresBatista}, twisted partial actions of Hopf algebras \cite{AlvesBatistaDokuPaques}, partial modules over a Hopf algebra \cite{AlvesBatistaVercruysse}, partial actions of multiplier Hopf algebras \cite{Fonseca}, partial groupoid actions on rings \cite{BagioPaques}, on $s$-unital rings \cite{BagioPinedo} and, very recently, on $R$-categories \cite{MarinPinedo}.
In this framework as well, having a globalization theorem triggered several new results. For example, in \cite{AlvesBatista} the authors obtained a version of Blattner-Montgomery theorem for the case of partial actions, extending the analogue of the Cohen-Montgomery duality obtained in \cite{Lomp}. 

However, in each of the previous situations there is an {\em ad hoc} construction of the globalization, depending heavily on the nature of the objects carrying the partial action.
In this paper we propose a unified approach to globalization in a categorical setting and we provide a concrete procedure to construct it. 

Our approach relies on the notion of geometric partial comodule recently introduced in \cite{JoostJiawey}. 
Unlike partial actions as described above, which exist only for (topological) groups and Hopf algebras, geometric partial comodules can be defined over any coalgebra in a monoidal category. Hence, their field of applications is much wider and, at the same time, it encompasses classical partial actions, which 
can be recovered by considering a group as a coalgebra in the opposite of the category of sets. Moreover, geometric partial comodules allow us to describe phenomena that are out of the reach of the theory of partial (co)actions in the Hopf algebra framework. For instance, recall that the coordinate algebras of algebraic groups provide classical examples of Hopf algebras, which in turn are the backbone of the algebraic approach to the representation theory of those, in the sense that regular actions of algebraic groups on affine varieties correspond to coactions of the coordinate Hopf algebras on the corresponding coordinate rings. Despite this, it has been shown in \cite{Dual} that a partial coaction in the sense of \cite{Caenepeel-Janssen} of the coordinate Hopf algebra $\cO(G)$ of an algebraic group $G$ on the coordinate ring $\cO(X)$ of an affine space $X$ is always global, unless $X$ is a disjoint union of non-empty subspaces. 

The notion of {\em geometric partial comodules} was proposed in \cite{JoostJiawey} as an alternative to partial (co)actions of Hopf algebras, in order to describe genuine (\eg irreducible) partial actions of algebraic groups from a Hopf-algebraic point of view. 
In view of this purpose, the prefix \emph{geometric} was added, in order to distinguish the latter ones from the \emph{algebraic} ones as in \cite{Caenepeel-Janssen}.
At the same time, however, it turned out that geometric partial comodules allow to approach in a unified way partial actions of groups on sets, partial coactions of Hopf algebras on algebras and partial (co)actions of Hopf algebras on vector spaces (i.e.\ partial (co)representations of Hopf algebras) as well. As a consequence, the question of studying the existence (and uniqueness) of globalization for geometric partial comodules naturally arises as a unifying way to address the issue. The present paper is devoted to deal with this question.

After recalling the main features of the theory of geometric partial comodules over coalgebras in \S\ref{ssec:genparcom}, in \S\ref{ssec:globcov} we recall the procedure to construct a geometric partial comodule from a global comodule $Y$ together with an epimorphism $p:Y\to X$ in the underlying category. The resulting geometric partial comodule $X$ is said to be induced by the pair $(Y,p)$. Moreover, by defining a suitable category of `covers', which are triples $(Y,X,p)$ as above, we show that this construction becomes functorial. 

Our main results are proven in \S\ref{ssec:globgen}, where we also introduce the concept of globalization for geometric partial comodules. In Theorem~\ref{thm:globalization}, definitely the most important of the paper, we provide necessary and sufficient conditions for the existence of the globalization and we exhibit an explicit construction of the latter, whenever it exists. In Corollary \ref{cor:adjunction}, this construction is proven to be functorial and to provide a right adjoint $\cG$ to the fully faithful embedding from global comodules into globalizable partial comodules.
A remarkable fact is that this adjunction provides a splitting of the classical free-forgetful adjunction for global comodules (see Proposition~\ref{prop:adj}), which shows once more how the theory of (geometric) partial comodules provides a refinement of the classical global theory.

Finally, in Theorem~\ref{globcov} we show that the category of globalizable geometric partial comodules is equivalent to the one of minimal proper covers, thus offering a way of concretely describing the globalizable partial comodules among all the geometric partial ones.

These results do not only provide an effective tool to compute the globalization whenever it exists, but they also allow to test if a globalization indeed exists and to provide an obstruction in case it does not. 
In fact, although globalizations of partial actions on topological spaces always exist, it is known that the topological properties of the initial space are not necessarily shared by its globalization (for example, the globalization of a partial action on a Hausdorff space is not necessarily Hausdorff). Similarly, partial actions on $C^*$-algebras are not globalizable in general (see \cite[Proposition 2.1]{Abadie} for a criterion for the existence of a globalization of a partial action on commutative $C^*$-algebras). Theorem~\ref{thm:globalization} allows to identify more cases of this phenomenon and, in particular, Corollary~\ref{obstruction} shows that in the category of algebras over a field there exist geometric partial comodules which do not admit a globalization. 
We conclude the paper with a few additional examples. 

In two forthcoming papers \cite{Saracco-Vercruysse1, Saracco-Vercruysse2}, we analyse closely a number of concrete instances of globalization obtained by applying the general approach of the present paper. In particular, we will show how globalization theorems appearing in the literature (and recalled above) are subsumed as particular instances of our results and, moreover, how our approach allows to obtain new types of globalizations as well.


\section{Geometric partial comodules}

\subsection{Preliminaries}\label{ssec:genparcom}

Let $\left(\cC,\otimes,\I,\calpha,\clambda,\crho\right)$ be a monoidal category with pushouts. For any object $X$ in $\cC$, we usually denote the identity morphism on $X$ again by $X$. Moreover, for any algebra $A$ and any coalgebra $H$ in $\cC$, we denote by $\Mod_A$ the category of (right) $A$-modules and by $\Com^H$ the category of (right) $H$-comodules. We also assume implicitly the category $\cC$ to be strict (\ie $\calpha,\clambda,\crho$ being identities) and hence omit the constraint isomorphisms.

Recall now from \cite[\S2]{JoostJiawey} the following definitions.

\begin{definition}
Let $(H,\Delta,\varepsilon)$ be a coalgebra in $\cC$. A \emph{partial comodule datum} is a cospan
\begin{equation}\label{eq:pcoaction}
\begin{gathered}
\xymatrix@R=7pt{
X \ar[dr]_-{\rho_X} & & X\otimes H \ar@{->>}[dl]^(0.45){\pi_X} \\
 & X\bullet H & 
}
\end{gathered}
\end{equation}
in $\cC$ where $\pi_X$ is an epimorphism.
\end{definition}

\begin{remark}\label{cospans}
Recall that cospans in a category with pushouts form a bicategory. The same is true for those cospans admitting a leg which is an epimorphism, as in \eqref{eq:pcoaction}. The composition is defined by means of the pushout of the adjacent maps, that is to say, for the cospans
\[
\begin{gathered}
\xymatrix@!C@R=7pt{
X_1 \ar[dr]_-{f_1} & & X_2 \ar@{->>}[dl]^{\pi_1} \\
 & Y_1 & 
}\end{gathered} \qquad \text{and} \qquad 
\begin{gathered}
\xymatrix@!C@R=7pt{
X_2 \ar[dr]_-{f_2} & & X_3 \ar@{->>}[dl]^{\pi_2} \\
 & Y_2 & 
}\end{gathered}
\]
the composition is provided by the cospan
\[
\xymatrix@!C@R=7pt{
X_1 \ar[dr]_-{f_1} & & X_2 \ar@{->>}[dl]^{\pi_1} \ar[dr]_-{f_2} & & X_3 \ar@{->>}[dl]^{\pi_2} \\
 & Y_1 \ar[dr]_-{\tilde{f}_2} &  & Y_2 \ar@{->>}[dl]^-{\tilde{\pi}_1} & \\
 & & Y_3.  \ar@{}[u]|<<<{\pushout} & & 
}
\]
Given two cospans $(Y_1,f_1,\pi_1)$ and $(Y_2,f_2,\pi_2)$ with same domain $X_1$ and same codomain $X_2$, a morphism $\alpha:(Y_1,f_1,\pi_1)\to (Y_2,f_2,\pi_2)$ of cospans is a morphism $\alpha: Y_1 \to Y_2$ in $\cC$ such that $\pi_2 = \alpha\circ \pi_1$ and $f_2 = \alpha \circ f_1$. Notice that if a morphism $\alpha$ as before exists, then it is unique and it is an epimorphism, because $\pi_1$ and $\pi_2$ are epimorphisms themselves. As a consequence, the Hom-categories in the bicategory of cospans with one epimorphic leg are in fact partially ordered sets. Thus, if between two such cospans there exist morphisms in both directions, then these morphisms are mutual inverses and so, in particular, isomorphisms.
\end{remark} 

Any partial comodule datum induces canonically the following pushouts 
\begin{equation}
\begin{gathered}
\xymatrix @!0 @R=37pt @C=55pt{
& X\ot H \ar@{->>}[dl]_-{\pi_X} \ar[dr]^-{\rho_X\ot H}\\
X\bul H \ar[dr]_(0.4){\rho_X\bul H}  && (X\bul H)\ot H \ar@{->>}[dl]^(0.4){\pi_{X\bul H}}\\
&(X\bul H)\bul H\ar@{}[u]|<<<{\pushout}
}
\end{gathered} 
\quad
\begin{gathered}
\xymatrix @!0 @R=37pt @C=55pt{
& X\ot H \ar@{->>}[dl]_-{\pi_X} \ar[dr]^-{X\ot \Delta}\\
X\bul H \ar[dr]_(0.4){X\bul \Delta} && X\ot H\ot H\ar@{->>}[dl]^(0.45){\pi_{X,\Delta}} \ar[dr]^(0.55){\pi_X\ot H}\\
& X\bul(H\ot H) \ar[dr]_(0.4){\pi'_{X}} \ar@{}[u]|<<<{\pushout}  && (X\bul H)\ot H \ar@{->>}[dl]^(0.4){\pi'_{X,\Delta}} \\
&&X\bul(H\bul H) \ar@{}[u]|<<<{\pushout}
}
\end{gathered}\label{eq:pushcan2}
\end{equation}

\begin{definition}
Let $(H,\Delta,\varepsilon)$ be a coalgebra in 
 $\cC$. A {\em geometric partial comodule} is a partial comodule datum $(X,X\bul H,\pi_X,\rho_X)$ that satisfies the following conditions.
\begin{enumerate}[label=({GP\arabic*}),leftmargin=1.3cm]
\item\label{item:QPC1} Counitality:
there exists a morphism $X\bul\varepsilon:X\bul H\to X$ which makes the following diagram commutative.
\[
\begin{gathered}
\xymatrix @!0 @R=16pt @C=65pt{
X  \ar@/_2.5ex/[dddr]_-{\id_X} \ar[dr]^-{\rho_X} && X\ot H  \ar@{->>}[dl]_-{\pi_X}\ar@/^2.5ex/[dddl]^-{X\ot\varepsilon} \\
& X\bul H \ar[dd]_-{X\bul\varepsilon} & \\
 & & \\
& X.
}
\end{gathered}
\]
\item\label{item:QPC2} Geometric coassociativity:
there exists an isomorphism 
\[\theta:X\bul(H\bul H)\to (X\bul H)\bul H\] 
such that the following diagrams commute
\[
\xymatrix@!0 @R=35pt @C=48pt{
& (X\bul H)\ot H \ar@{->>}[dl]_(0.6){\pi'_{X,\Delta}} \ar@{->>}[dr]^(0.6){\pi_{X\bul H}}\\
X\bul (H\bul H) \ar[rr]_-{\theta} && (X\bul H)\bul H
}\qquad
\xymatrix @!0 @R=35pt @C=48pt{
X \ar[rr]^-{\rho_X} \ar[d]_-{\rho_X} &&X\bul H \ar[rr]^-{\rho_X\bul H} && (X\bul H)\bul H 
\\
X\bul H \ar[rr]_-{X\bul \Delta} && X\bul(H\ot H) \ar[rr]_-{\pi'_X} && X\bul(H\bul H). \ar[u]_{\theta} 
}
\]
\end{enumerate}

If $(X,X\bul H,\pi_X,\rho_X)$ and $(Y,Y\bul H,\pi_Y,\rho_Y)$ are geometric partial comodules, then a \emph{morphism of geometric partial comodules} is a pair $(f,f\bul H)$ of morphisms in $\cC$ with $f:X\to Y$ and $f\bul H: X\bul H \to Y\bul H$ such that the following diagram commutes
\begin{equation}\label{eq:shield}
\begin{gathered}
\xymatrix @!0 @R=18pt @C=60pt {
X \ar[dd]_-{f} \ar[dr]^-{\rho_X} & & X\ot H \ar@{->>}[dl]_-{\pi_X} \ar[dd]^-{f\ot H} \\ 
 & X\bul H \ar[dd]_-{f\bul H} &  \\
Y \ar[dr]_-{\rho_Y} & & Y \ot H \ar@{->>}[dl]^-{\pi_{Y}} \\
 & Y \bul H. & 
}
\end{gathered}
\end{equation}
We will often denote a geometric partial comodule $(X,X\bul H,\pi_X,\rho_X)$ simply by $X$ and a morphism as above simply by $f$. Moreover, we denote by $\gparcom{H}$ the category of geometric partial comodules over $H$.
\end{definition}

From the above definitions (and more precisely from the fact that the morphisms $\pi_X$ are epimorphisms) it follows that the obvious forgetful functor $U:\gparcom{H}\to \cC$ is faithful. 
In addition, any usual (global) comodule $(X,\delta_X)$ over $H$ is a geometric partial comodule where $\pi_X$ is the identity and $\rho_X \coloneqq \delta_X$. More precisely, $\Com^H$ is a full subcategory of $\gparcom{H}$ and we denote the associated embedding functor by $\cI:\com{H} \to \gparcom{H}$.

By specializing $\cC$ to appropriate categories, examples of geometric partial comodules can be obtained from various partial structures studied extensively in literature, such as partial actions of (topological) groups and monoids (see \cite{Abadie,KellendorkLawason,Novikov}), partial (co)actions and (co)representations of Hopf algebras (see \cite{ParCorep,AlvesBatistaVercruysse}) and partial comodule algebras (see \cite{AlvesBatista,Caenepeel-Janssen}). 
For some concrete examples we refer the reader to \cite{JoostJiawey}, \cite{Saracco-Vercruysse1} and \cite{Saracco-Vercruysse2}.

\begin{remarks}\label{rem:strict}
\begin{enumerate}[label=(\roman*), ref=(\roman*), leftmargin=0.8cm] 
\item\label{item:rems1} The notion of geometric partial comodule should not be confused with the notion of partial comodule over a Hopf algebra as it appears in \cite[\S3]{ParCorep}. 
\item\label{item:rems2} If $(Y,\delta_Y)$ is an $H$-comodule, viewed as a geometric partial comodule under the embedding functor $\cI$, and if $(X,X\bul H,\pi_X,\rho_X)$ is a geometric partial comodule, then $f:Y\to X$ is a morphism of geometric partial comodules $\cI(Y) \to X$ if and only if the following diagram commutes
\begin{equation}\label{eq:morphismIPCD}
\begin{gathered}
\xymatrix{
Y \ar[rr]^f \ar[d]_-{\delta_Y} && X \ar[d]^-{\rho_X}\\
Y\ot H \ar[r]^-{f\ot H} & X\ot H \ar[r]^-{\pi_X} & X\bul H.
}
\end{gathered}
\end{equation}
\end{enumerate}
\end{remarks}

Let us conclude this subsection by recalling that one of the most important sources of geometric partial comodules is provided by the so-called \emph{induction procedure}. This construction appears originally as \cite[Example 2.5]{JoostJiawey} under slightly stronger hypotheses on the base category $\cC$, but the argument in \cite{JoostJiawey} still holds in the present context, too.

\begin{definition}\label{def:induction}
Let $(Y,\delta)$ be an $H$-comodule and let $p:Y\to X$ an epimorphism in $\cC$. 
The pushout
\begin{equation}\label{eq:globcom}
\begin{gathered}
\xymatrix @!0 @R=23pt @C=60pt {
 & Y \ar[dr]^(0.55){\,\ (p\otimes H)\circ\delta} \ar@{->>}[dl]_-{p} & \\
 X \ar[dr]_-{\rho_X} & & X\otimes H \ar@{->>}[dl]^-{\pi_X} \\
 & X\bul H \ar@{}[u]|<<<{\pushout} &
}
\end{gathered}
\end{equation}
makes $(X,X\bul H,\pi_X,\rho_X)$ a geometric partial comodule and $p$ becomes a morphism of geometric partial comodules $p: \cI(Y) \to X$. We refer to this as the \emph{induced partial comodule} structure from $Y$ to $X$.
\end{definition}

The motivation for the above construction comes from the following example.

\begin{example}\label{ex:induced}
Considering the case $\cC=\op{\Set}$,
assume that $Y$ is a $G$-set with global action $\beta:G\times Y\to Y$ and that $j:X\subseteq Y$ is any subset. One can perform the pullback
\[
\xymatrix @!0 @R=23pt @C=60pt {
 & Y & \\
G\times X \ar[ur]^(0.55){\beta\circ(G\times j)\ } & & X \ar@{_{(}->}[ul]_-{j} \\
 & G\bul X. \ar@{_{(}->}[ul]^-{\iota} \ar[ur]_-{\alpha} \ar@{}[u]|<<<{\pushout} &
}
\]
Because $j$ is injective, it turns out that $G\bul X=\left\{(g,x)\in G\times X\mid \beta_g(x) \in X\right\}$ and that $\alpha(g,x) = \beta_g(x)$ for all $x\in X\cap\beta_g^{-1}(X)$. If we define $X_{\gi{g}} = \{x\in X\mid (g,x)\in G\bullet X\}$ and $\alpha_g(x) \coloneqq  \alpha(g,x)$ for all $g\in G$, then $\{X_g,\alpha_g\}$ gives a partial action of $G$ on $X$ in the sense of {\cite[Definition 1.2]{Exel-set}}. 
We say that this is the partial action {\em induced} from $Y$ to $X$. 
\end{example}

The so-called \emph{globalization question} concerns exactly the problem of deciding when a geometric partial comodule structure on an object $X$ over a coalgebra $H$ has been induced by a (preferably, uniquely determined) global comodule $(Y,\delta)$ as in Definition \ref{def:induction} and which geometric partial comodules admit such an inducing global comodule. In the present paper we will address both these questions by providing a criterion to determine when a geometric partial comodule is globalizable (Theorem \ref{thm:globalization}) and by providing a complete description (under some mild assumptions on the category of $H$-comodules) of the globalizable geometric partial comodules, in terms of what we are going to call in the next subsection the \emph{minimal proper covers} (Theorem \ref{globcov}).

\subsection{Making induction functorial: the category of global covers}\label{ssec:globcov}

Inspired by the construction from Definition \ref{def:induction}, we introduce the following definition, which allows us to make the process of induction functorial.

\begin{definition}\label{def:globcov}
We denote by $\Cov ^H$ the category whose objects are triples $(Y,X,p)$ where $Y$ is a global $H$-comodule, $X$ is an object in $\cC$ and $p:Y\to X$ is an epimorphism in $\cC$. We will refer to these objects as (\emph{global}) \emph{covers} and often denote them simply by $p: Y \to X$. 

A morphism $(F,f):(Y,X,p)\to (Y',X',p')$ in $\Cov ^H$ consists of  a morphism of $H$-comodules $F:Y\to Y'$ and a $\cC$-morphism $f:X\to X'$ such that $p'\circ F=f\circ p$.
\end{definition}

\begin{proposition}\label{pr:covers}
The procedure of constructing the induced geometric partial comodule structure from a global cover as in Definition \ref{def:induction} defines a functor
\[
\Ind:\Cov^H\to \gparcom{H}.
\]
\end{proposition}

\begin{proof}
From Definition \ref{def:induction} we see that $\Ind$ is well-defined on objects. If $(F,f):(Y,X,p)\to (Y',X',p')$ is a morphism in $\Cov^H$, then we set $\Ind(F,f) \coloneqq f$. Since $F$ is $H$-colinear,
\begin{align*}
\pi_{X'}\circ (f\ot H)\circ (p\ot H)\circ \delta_Y & = \pi_{X'}\circ (p'\ot H)\circ (F\ot H)\circ \delta_Y = \pi_{X'}\circ (p'\ot H)\circ \delta_{Y'}\circ F\\
&\stackrel{\eqref{eq:globcom}}{=} \rho_{X'}\circ p'\circ F = \rho_{X'}\circ f\circ p
\end{align*}
and therefore, by the universal property of the pushout $X\bul H$, there exists a unique morphism $f\bul H:X\bul H\to X'\bul H$ in $\cC$ such that $(f\bul H)\circ \rho_X=\rho_{X'}\circ f$ and $(f\bul H)\circ \pi_X=\pi_{X'}\circ (f\ot H)$, i.e. $(f,f\bul H)$ is a morphism in $\gparcom{H}$. 
\end{proof}

The following example shows that the same geometric partial comodule can be induced by many different global comodules. 

\begin{example}\label{ex:global}
Assume that $X$ and $Y$ are global $H$-comodules and that $p:Y \to X$ is an epimorphism which is $H$-colinear. 
One may check directly that we have the pushout
\[
\begin{gathered}
\xymatrix @!0 @R=23pt @C=60pt {
 & Y \ar[dr]^(0.55){(p\otimes H)\circ\delta_Y} \ar@{->>}[dl]_-{p} & \\
 X \ar[dr]_-{\delta_X} & & X\otimes H \ar@{=}[dl] \\
 & X\ot H \ar@{}[u]|<<<{\pushout} &
}
\end{gathered}
\]
and hence the induced geometric partial comodule structure on $X$ is its own global one.
\end{example}

In view of Example \ref{ex:global}, it is natural to introduce the following definitions, in order to avoid that a cover contains superfluous information.

\begin{definition}\label{def:cogen}
Let $(Y,\delta_Y)$ be a global comodule over the coalgebra $H$ in $\cC$ and let $p:Y\to X$ be an epimorphism in $\cC$. We say that $Y$ is \emph{co-generated} by $X$ as a comodule if the following composition is a monomorphism in $\com{H}$:
\[
(p\ot H)\circ \delta_Y:Y\to X\ot H.
\]
\end{definition}

\begin{example}\label{ex:cogen}
The subsequent examples argue in favour of the appropriateness of the terminology just introduced. 
\begin{enumerate}[label=(\arabic*),ref=(\arabic*),leftmargin=1cm]
\item\label{cog:item1} If we rephrase Definition \ref{def:cogen} for $\cC=\op{\vectk}$, then we recover the classical notion of module generated by a subspace.

\item\label{cog:item2} If we let $\cC$ be $\op{\Set}$ and $H$ be a group, then we recover the familiar notion of orbit under the action of $H$.

\item In $\cC=\vectk$ instead, it says exactly that $Y$ is isomorphic to a subcomodule of the free comodule $X\ot H$ (via $(p\ot H)\circ \delta_Y$), which is in accordance with the definition of finitely co-generated comodules in $\vectk$ used in \cite[Example 1.2]{Takeuchi}. 
\end{enumerate}
\end{example}

\begin{definition}\label{def:minprop}
A global cover $(Y,X,p)$ is called {\em proper} if $Y$ is co-generated by $X$, that is, if $(p \ot H) \circ \delta_Y$ is a monomorphism in $\com{H}$. Denote by $\Cov^H_{pr}$ the full subcategory of $\Cov^H$ consisting of all proper covers. 

A proper cover $(Y,X,p)$ is called {\em minimal} if it does not factor through another proper cover. More explicitly, if $(Y',X,p')$ is another proper cover such that $p=p'\circ q$ for some morphism $q:Y\to Y'$ in $\cC$, then $q$ is an isomorphism. Denote by $\Cov^H_{pr,min}$ the full subcategory of $\Cov^H_{pr}$ consisting of all minimal proper covers. 
\end{definition}

\begin{corollary}\label{co:covers}
The functor $\Ind$ from Proposition \ref{pr:covers} restricts to a faithful functor $\Ind:\Cov^H_{pr}\to \gparcom{H}$.
\end{corollary}

\begin{proof}
Consider morphisms $(F,f),(F',f):(Y,X,p)\to (Y',X',p')$ in $\Cov^H_{pr}$, so that $\Ind(F,f)=f=\Ind(F',f)$. Then by a similar computation as in the proof of Proposition~\ref{pr:covers}, we find that
$$(p'\ot H)\circ \delta_{Y'}\circ F=(f\ot H)\circ (p'\ot H)\circ \delta_Y = (p'\ot H)\circ \delta_{Y'}\circ F'.$$
Since $(Y',X',p')$ is proper, $(p'\ot H)\circ \delta_{Y'}$ is a monomorphism in $\com{H}$, hence $F=F'$, hence $(F,f)=(F',f)$ and $\Ind$ is faithful.
\end{proof}

\begin{example}\label{ex:TopCompl}
It is important to notice that the functor $\Ind$ from Corollary~\ref{co:covers} is not full. For example, in $\cC = \op{\Top}$ consider the topological group $H \coloneqq \left((\R,\top),+,0\right)$ and the natural (topological) global action $\rho_X$ of $H$ on $X \coloneqq (\R,\top)$ itself given by translation, where all the copies of $\R$ have the ordinary euclidean topology $\top$. Then $\left(X,X,\id_\R\right)$ is an object in $\Cov^{H}_{pr}$ and $\Ind(X,X,\id_\R)=(X,H \times X,\id_{H \times X},\rho_X)$ with the global action of $H$.

Consider also the global action of $H=((\R,\top),+,0)$ on $Y \coloneqq (\R,\top')$, but with the trivial topology $\top'\coloneqq \{\emptyset,\R\}$. The identity $p\coloneqq \id_\R:(\R,\top)\to (\R,\top')$ is a continuous $\R$-linear monomorphism. The triple $(Y,X,p)$ is an object in $\Cov^{H}_{pr}$, because $H \times X \xrightarrow{\id_\R \times p} H \times Y \to Y$ is a continuous epimorphism. The geometric partial comodule $\Ind(Y,X,p)$ in this case is again simply the global action $(X,H \times X,\id_{H \times X},\rho_X)$ of $H$ on $X$.

Therefore, $(\id_X,\id_{H \times X})\in \gparcom{H}(\Ind(X,X,\id_\R),\Ind(Y,X,p))=\gparcom{H}(X,X)$, but since the identity map on $\R$ is not a continuous morphism from $Y$ to $X$, the morphism $(\id_X,\id_{H \times X})$ is not in the image of the functor $\Ind$, hence the functor is not full. 

\end{example}

Remark however that the cover $(Y,X,p)$ in Example \ref{ex:TopCompl} is not minimal: $Z \coloneqq (\R,\top'')$ where $\top''$ is any intermediate topology $\top'\subsetneq \top'' \subsetneq \top$ provides an intermediate, non homeomorphic, proper cover. In \S\ref{preglob}, and under some mild assumptions, we will prove that the restriction of the functor $\Ind$ to the category of minimal proper covers {\em is} full, by means of the globalization procedure (see Theorem~\ref{globcov}).


\section{The globalization question}\label{sec:globque}

As mentioned at the end of \S\ref{ssec:genparcom}, the globalization question concerns the problem of determining when a geometric partial comodule structure is induced by a (unique) global one and how to describe the induced partial comodules among all the geometric partial ones. We begin by addressing the first problem.

\subsection{Globalization for geometric partial comodules}\label{ssec:globgen}

As the globalization of a partial action of a group $G$ on a set $X$ is the smallest $G$-set containing $X$ and such that the partial action is induced by restriction of the global one (see \cite[Theorem 1.1]{Abadie}), we expect the globalization of a partial comodule $X$ to be a universal  $H$-comodule ``covering'' $X$ and such that the partial coaction is induced by the global one.

\begin{definition}\label{def:glob}
Given a geometric partial comodule $(X,X\bul H,\pi_X,\rho_X)$ over the coalgebra $H$ in the monoidal category with pushouts $\cC$, a \emph{globalization} for $X$ is a global comodule $(Y,\delta)$ with a morphism $p:Y\to X$ in $\cC$ such that
\begin{enumerate}[label=(GL\arabic*),ref=(GL\arabic*)]
\item\label{item:GL0} $p:\cI(Y) \to X$ is a morphism of geometric partial comodules (that is, \eqref{eq:morphismIPCD} commutes);
\item\label{item:GL2} the corresponding diagram \eqref{eq:globcom} 
is a pushout square in $\cC$;
\item\label{item:GL3} $Y$ is universal among all global comodules admitting a morphism of geometric partial comodules to $X$: if $(Z,\delta')$ is global and $q:\cI(Z)\to X$ is of partial comodules, then there is a unique morphism of global comodules $\eta:Z\to Y$ such that $p\circ\eta = q$.
\end{enumerate}
We say that $X$ is {\em globalizable} if a globalization for $X$ exists and we denote by $\gparcom{H}_{gl}$ the full subcategory of $\gparcom{H}$ composed by the globalizable partial comodules. 
\end{definition}

\begin{lemma}\label{lem:hihi}
Let $(X,X\bul H,\pi_X,\rho_X)$ be a geometric partial comodule and $(Y,\delta)$ be a global comodule. If $p:Y\to X$ is a morphism of geometric partial comodules in $\cC$ such that \eqref{eq:globcom} is a pushout diagram, then $p$ is an epimorphism.
\end{lemma}

\begin{proof}
If $f,g : X \to S$ are two morphisms in $\cC$ such that $f \circ p = g \circ p$, then
\[g \circ (X \ot \varepsilon) \circ (p \ot H) \circ \delta = g \circ p = f \circ p\]
and hence, by the universal property of the pushout, there exists a unique morphism $\sigma : X \bul H \to S$ in $\cC$ such that $\sigma \circ \rho_X = f$ and $\sigma \circ \pi_X = g\circ (X \ot \varepsilon)$. However, the counitality condition \ref{item:QPC1} entails that
$\sigma \circ \pi_X = g\circ (X \ot \varepsilon) = g\circ (X \bul \varepsilon) \circ \pi_X$
and hence $\sigma = g\circ (X \bul \varepsilon)$. Thus, by \ref{item:QPC1} again,
$f = \sigma \circ \rho_X = g\circ (X \bul \varepsilon) \circ \rho_X = g$.
\end{proof}

Let $(X,X\bul H,\pi_X,\rho_X)$ be a geometric partial comodule over $H$. 
Axioms \ref{item:GL0}, \ref{item:GL2} and Lemma \ref{lem:hihi} tell us that the partial comodule $X$ is induced by the global comodule $Y$ as in Definition \ref{def:induction}. 
Axiom \ref{item:GL3} ensures that $Y$ does not carry superfluous information, as it is clear that if $p':Y'\to Y$ is an epimorphism of global comodules, then $X$ is induced by $p\circ p':Y'\to X$ as well. The universal property \ref{item:GL3} assures also that a globalization is unique up to isomorphism, whenever it exists.

Therefore, the globalization of $X$ is, by definition, a global cover $(Y,X,p)$ in the sense of Definition \ref{def:globcov} such that the given geometric partial comodule structure on $X$ is induced by the global comodule structure on $Y$ and such that $Y$ is universal with respect to this property. This suggests that one may call an \emph{inducing global cover} a global comodule $Y$ satisfying \ref{item:GL0} and \ref{item:GL2} and then the globalization would be the universal inducing global cover, in the sense of condition \ref{item:GL3}.

Similarly, one may observe that conditions \ref{item:GL0} and \ref{item:GL3} can be treated independently from condition \ref{item:GL2} (see \cite[Remark 2.4]{Saracco-Vercruysse1}). This suggests that one may also call \emph{pre-globalization} a global module $(Y, \delta)$ together with a morphism $p : Y \to X$ in $\cC$ satisfying \ref{item:GL0} and \ref{item:GL3} and not necessarily \ref{item:GL2}. 

However, since we are interested in globalizations as they have been dealt with in the literature (that is, which are inducing the given partial comodule structure and which are universal with respect to this property), we focus on global comodules satisfying all the conditions \ref{item:GL0}-\ref{item:GL3} at the same time.

The following lemma represents the key step toward our globalization theorem.

\begin{lemma}\label{constructionpreglob}
Let 
$(H,\Delta,\varepsilon)$ be a coalgebra in $\cC$. Consider a geometric partial comodule $(X,X\bul H,\pi_X,\rho_X)$, the associated free (global) $H$-comodule $(X\ot H,X\ot \Delta)$ and the parallel morphisms
\begin{equation}\label{eq:glob1}
\xymatrix@C=75pt{
(X\otimes H,X\ot \Delta) \ar@<+0.8ex>[r]^-{\rho_X\ot H} \ar@<-0.4ex>[r]_-{(\pi_X\ot H)\circ(X \otimes \Delta)} &  (X\bul H\otimes H, X\bul H\ot \Delta)
}
\end{equation}
in $\com{H}$. Consider as well a global $H$-comodule $(Y,\delta)$ and the induced geometric partial comodule $\cI(Y)=(Y,Y\ot H,\id,\delta)$. 
Then there is a bijective correspondence: 
\begin{align*}
\gparcom{H}(\cI(Y),X) & \cong \{f\in \com{H}(Y,X\ot H) ~|~ (\rho_X\ot H)\circ f=(\pi_X\ot H)\circ (X\ot \Delta)\circ f\}\\
g &\mapsto (g\ot H)\circ \delta \\
(X\ot \varepsilon)\circ f & \mapsfrom f
\end{align*}
Moreover, this correspondence is natural in both arguments $Y$ and $X$.
\end{lemma}

\begin{proof}
Consider $g\in \gparcom{H}(\cI(Y),X)$, that is to say, $\pi_X\circ\left(g\ot H\right) \circ \delta = \rho_X \circ g$ (see \eqref{eq:morphismIPCD}). Since both $\delta: Y\to Y\ot H$ and $g\ot H:Y\ot H\to X\ot H$ are morphisms of global $H$-comodules, so is $(g\ot H)\circ \delta$. Moreover, we find that 
\begin{multline*}
\left(\rho_X\ot H\right) \circ \left(g\ot H\right) \circ \delta \stackrel{\eqref{eq:morphismIPCD}}{=} \left(\pi_X \ot H\right) \circ \left(g\ot H \ot H\right) \circ \left(\delta \ot H\right) \circ \delta \\
 = \left(\pi_X \ot H\right) \circ \left(g\ot H \ot H\right) \circ \left(Y \ot \Delta\right) \circ \delta = \left(\pi_X \ot H\right) \circ \left(X \ot \Delta\right) \circ \left(g\ot H\right) \circ \delta.
\end{multline*}
Consequently, the first map of the statement is well-defined. 

Conversely, for any morphism of global comodules $f:Y\to X\ot H$ equalizing $(\rho_X\ot H)$ and $(\pi_X\ot H)\circ (X\ot \Delta)$ we find, using this equalizing property in the second equality, that 
\begin{eqnarray*}
\rho_X\circ (X\ot \varepsilon)\circ f &=& (X\bul H\ot \varepsilon)\circ (\rho_X\ot H)\circ f\\
&=& (X\bul H\ot \varepsilon)\circ (\pi_X\ot H)\circ (X\ot \Delta)\circ f\\
&=& \pi_X \circ (X\ot H\ot \varepsilon) \circ (X\ot \Delta)\circ f = \pi_X \circ f\\
&=& \pi_X \circ (X\ot \varepsilon \ot H) \circ (X\ot \Delta)\circ f\\
&=& \pi_X \circ (X\ot \varepsilon \ot H) \circ (f\ot H)\circ \delta
\end{eqnarray*}
where we used the $H$-colinearity of $f$ in the last equality. Consequently (see \eqref{eq:morphismIPCD}), $(X\ot \varepsilon)\circ f$ is a morphism of geometric partial comodules.

As $f \in \com{H}(Y,X\ot H)$ is $H$-colinear, we have that $f=(X\ot \varepsilon \ot H) \circ (f\ot H)\circ \delta$.  Finally, for any $g\in \gparcom{H}(\cI(Y),X)$ we obviously have
$$(X\ot \varepsilon)\circ (g\ot H)\circ \delta=g\circ (Y\ot \varepsilon)\circ \delta=g$$
and hence we obtain the required bijection.
Naturality follows by a direct computation.
\end{proof}

In case $\cC=\Set^{op}$, Definition \ref{def:glob} coincides with the globalization (or enveloping action) as defined and studied in \cite{Abadie}. It has been proven, for instance in \cite[Theorem 1.1]{Abadie} or \cite[\S3.1]{KellendorkLawason}, that the globalization of a partial action of a group $G$ on a set $X$ or, equivalently, of a geometric partial comodule
\[
\xymatrix @!0 @C=55pt @R=25pt {
G \times X & & X \\
 & G\bullet X\ar[ur]_{\alpha} \ar@{_{(}->}[ul]^{\iota} & 
}
\] 
over $G$ in $\Set^{op}$ (see \cite{JoostJiawey} and Example \ref{ex:induced}), always exists and it is given by the quotient $Y\coloneqq (G\times X)/\sim$, where $(g,x)\sim(h,y)$ if and only if $(h^{-1}g,x)\in G\bullet X$ and $\alpha(h^{-1}g,x)=y$. The (global) action of $G$ on $Y$ is given by 
$$h\cdot [g,x]=[hg,x],$$
where $[g,x]$ denotes the equivalence class of $(g,x)\in G\times X$ under the relation $\sim$. The following proposition clarifies the motivation behind the approach to globalization we advocate in this paper.

\begin{proposition}\label{prop:globset}
Consider a group $G$ and a partial action $(X,G\bul X,\iota,\alpha)$. Then the globalization of $X$ is given exactly by the 
the coequalizer in $\Set$ of the pair 
\begin{equation}\label{eq:sets}
\begin{gathered}
\xymatrix @C=70pt {
G\times (G\bullet X) \ar@<+0.4ex>[r]^-{G\times\alpha} \ar@<-0.4ex>[r]_-{(\mu\times X)\circ(G\times \iota)} & G\times X.
}
\end{gathered}
\end{equation}
\end{proposition}

\begin{proof}
From the remark preceding this proposition we know that the globalization $G\times X/\sim$, being a quotient by an equivalence relation, is by definition the coequalizer of the pair $\xymatrix{R \ar@<+0.4ex>[r]^-{p_1} \ar@<-0.4ex>[r]_-{p_2} & G\times X}$ where
\[
R \coloneqq \left\{((g,x),(h,y))\in (G\times X)\times (G\times X)\mid x\in X_{g^{-1}h} \ \text{and} \ \alpha_{h^{-1}g}(x)=y\right\}
\]
is the equivalence relation $\sim$ and $p_1,p_2$ are the (restrictions of the) canonical projections. One may check that the assignments
\begin{gather*}
\varphi:R \to G \times (G\bul X), \quad ((g,x),(h,y)) \mapsto (h,(h^{-1}g,x)), \\
\psi:G \times (G\bul X)\to R, \quad (m,(n,z))\mapsto ((mn,z),(m,n\cdot z)),
\end{gather*}
are well-defined and each other inverses, making the following diagram
\[
\begin{gathered}
\xymatrix @C=60pt {
R \ar@<+0.5ex>[r]^-{p_1} \ar@<-0.5ex>[r]_-{p_2} \ar@<+0.4ex>[d]^{\varphi} & G \times X \ar@{=}[d] \\
G \times (G\bul X) \ar@<+0.5ex>[r]^-{(\mu \times X)\circ (G\times \iota)} \ar@<-0.5ex>[r]_-{G \times \alpha} \ar@<+0.4ex>[u]^{\psi}  & G \times X
}
\end{gathered}
\]
to commute sequentially. Hence, it turns out that $Y$ together with the obvious projection $p:G\times X\to Y,\, (g,x)\mapsto [g,x],$ is the coequalizer in $\Set$ of the pair \eqref{eq:sets}.
\end{proof}

Lemma \ref{constructionpreglob} and Proposition~\ref{prop:globset} suggest that the globalization for a partial comodule could be constructed by considering the equalizer of the corresponding pair \eqref{eq:glob1}. Our main result, Theorem \ref{thm:globalization}, shows that this is indeed the case.

As we henceforth need that a particular equalizer in the category of comodules over the coalgebra $H$ exists, let us recall that this is a rather mild condition. In fact, in \cite{Porst} it is shown that for $\cC$ of the form $\Mod_k$ (for a commutative ring $k$) or ${_A\Mod_A}$ (for a possibly non-commutative ring $A$), the category of comodules over any coalgebra in $\cC$ is complete.
Furthermore, it is well-known that the limit of any given diagram in $\Com^H$ exists whenever the limit of the same diagram in $\cC$ exists and the functor $-\ot H\ot H:\cC\to \cC$ preserves it. In particular, 
\begin{center}
$\Com^H$ is complete if $\cC$ is complete and $H$ is a flat object in $\cC$,
\end{center} 
that is, when the endofunctor $-\ot H$ preserves limits. These observations can be deduced from, for example, \cite[Proposition 4.3.2]{BorII}.

The advantage of the last case is that limits can be computed in the underlying category $\cC$.  Examples of such categories are $(\Set,\times,\{*\})$, $(\op{\Set},\times,\{*\})$, $(\vectk,\otimes_\K,\K)$, $(\op{\vectk},\otimes_\K,\K)$ and $(\Calgk,\otimes_\K,\K)$ where $\K$ is a field, or the category $(\Rmod{k},\otimes_k,k)$ of (symmetric) modules over a commutative ring $k$, provided that the coalgebra $(H,\Delta,\varepsilon)$ is such that $H$ is flat as $k$-module. 
The category $(\op{\Top},\times,\{*\})$ is an example as well, provided that the monoid $(H,\mu,u)$ in $\Top$ is such that $H$ is locally compact Hausdorff (a sufficient condition to have that $- \times H$ preserves colimits in $\Top$) or that $H$ is a topological group. The explicit globalization for all these cases will be treated separately, in details, in \cite{Saracco-Vercruysse1, Saracco-Vercruysse2}.

\begin{theorem}\label{thm:globalization}
Let $H$ be a coalgebra in the monoidal category with pushouts $\cC$. 
Then a geometric partial $H$-comodule $X=(X,X\bul H,\pi_X,\rho_X)$ is globalizable if and only if 
\begin{enumerate}[label=(\Roman*),ref=(\Roman*),leftmargin=1cm]
\item\label{item:glob1} the equalizer $(Y_X,\kappa)$ of the pair \eqref{eq:glob1} exists in $\com{H}$ and
\item\label{item:glob3} the commutative diagram 
\begin{equation}\label{eq:GXglob}
\begin{gathered}
\xymatrix @!0 @R=23pt @C=60pt {
 & Y_X \ar[dr]^-{\kappa} \ar@{->>}[dl]_-{(X\ot \varepsilon)\circ \kappa} & \\
 X \ar[dr]_-{\rho_X} & & X\otimes H \ar@{->>}[dl]^-{\pi_X} \\
 & X\bul H & 
}
\end{gathered}
\end{equation}
is a pushout diagram in $\cC$. 
\end{enumerate}
Moreover, under these equivalent conditions $Y_X$ is the globalization of $X$ and this globalization is co-generated by $X$ as a global $H$-comodule.
\end{theorem}

\begin{proof}
Observe that, by definition, $X$ is globalizable if and only if there exists a universal arrow $\big((Y,\delta),p\big)$, in the sense of \cite[\S{III.1}]{MacLane}, from $\cI$ to $X$ (conditions \ref{item:GL0} and \ref{item:GL3}) such that \ref{item:GL2} holds. Therefore, if $X$ is globalizable then the assignment
\begin{align*}
\com{H}(Z,Y) & \to \gparcom{H}(\cI(Z),X), \\ 
\eta & \mapsto p \circ \eta,
\end{align*}
is bijective. Hence, by Lemma \ref{constructionpreglob}, $Y$ with $\kappa \coloneqq (p \ot H) \circ \delta$ is the equalizer of $\rho_X \ot H$ and $(\pi_X \ot H) \circ (X \ot \Delta)$ in $\com{H}$ and so \ref{item:glob1} holds. 
Moreover, since $(X\ot \varepsilon)\circ \kappa=p$ and $\kappa=(p\ot H)\circ \delta$, diagram \eqref{eq:GXglob} coincides with diagram \eqref{eq:globcom}, which is a pushout by \ref{item:GL2}.

Conversely, if the equalizer $\big((Y_X,\delta),\kappa\big)$ of \eqref{eq:glob1} exists in $\com{H}$ then the assignment
\begin{align*}
\com{H}(Z,Y_X) & \to \{f\in \com{H}(Y_X,X\ot H) \,|\, (\rho_X\ot H)\circ f=(\pi_X\ot H)\circ (X\ot \Delta)\circ f\} \\ 
\eta & \mapsto \kappa \circ \eta,
\end{align*}
is bijective and hence, by Lemma~\ref{constructionpreglob} again, we have that $\epsilon \coloneqq (X\ot \varepsilon)\circ \kappa: \cI(Y_X) \to X$ is a universal arrow from $\cI$ to $X$ (i.e., \ref{item:GL0} and \ref{item:GL3} hold).
Since property \ref{item:glob3} in the statement of the theorem is exactly axiom \ref{item:GL2}, $Y_X$ is a globalization of $X$.

For the last statement, it follows from the above that the globalization of $X$ is given by $Y_X$. Moreover, since $\kappa=(\epsilon\ot H)\circ \delta:Y_X\to X\ot H$ is an equalizer, it is a monomorphism, while $\epsilon=(X\ot \varepsilon)\circ \kappa: Y_X \to X$ is an epimorphism by Lemma \ref{lem:hihi}. 
Therefore, by Definition \ref{def:cogen}, $Y_X$ is co-generated by $X$ as global comodule.
\end{proof}

It is of fundamental importance to realize that the necessary and sufficient conditions of Theorem \ref{thm:globalization} are not always satisfied in general, contrarily to what happens for partial actions of groups. In the next example, we show a case where they fail to be fulfilled. 

\begin{example}\label{ex:NO}
Let $\cC=\algk$, the category of algebras over a field $\K$. Take $H\coloneqq \K\left[x\right]$, the monoid bialgebra over $\N$ with unit $u:\K\to \K[x]$, $\Delta(x) = x\otimes x$ and $\varepsilon(x) = 1$, and take $A\coloneqq \K$ itself. Set $A\bul H\coloneqq \K\left[y^{\pm1}\right]$. The canonical inclusion $\pi:\K[x] \to \K\left[y^{\pm1}\right], x \mapsto y$, is an epimorphism of algebras. Therefore, the cospan
\[
\xymatrix @!0 @C=55pt @R=23pt{
\K \ar[dr]_-{u'} & & \K[x] \ar@{->>}[dl]^-{\pi} \\
 & \K\left[y^{\pm1}\right] & 
}
\] 
is a partial comodule datum in $\algk$, where $u': \K\to \K\left[y^{\pm 1}\right]$ is the unit of $\K\left[y^{\pm 1}\right]$.  We claim that $\left(\K,\K[y^{\pm1}],\pi,u'\right)$ is, in fact, a geometric partial comodule. 
First of all, observe that $\K\left[y^{\pm 1}\right]$ is the group Hopf algebra $\K[\Z]$ with counit $\varepsilon'$ and comultiplication $\Delta'$ and that $\pi:\K[x] \to \K[y^{\pm 1}]$ is a bialgebra morphism. This means that $\rho \coloneqq u : \K \to \K[y^{\pm 1}]$ is simply the extension of scalars of $\delta \coloneqq u : \K \to \K[x]$ along $\pi$ (and this intuitively justifies why it is coassociative and counital). 
Then, to compute $(A \bul H) \bul H$ one observes that if $B$ is a $\K$-algebra and $f:\K[y^{\pm 1}] \to B$, $g:\K[y^{\pm1}] \otimes \K[x] \to B$ are morphism of algebras such that $g\circ (u \otimes \K[x]) = f \circ \pi$ then: \begin{enumerate*}[label=(\roman*)] \item $g(1 \ot x) = f(y) \in B^{\times}$, the invertible elements in $B$, and hence there exists a unique $\tilde{g}:\K[y^{\pm 1}] \to B$ of $\K$-algebras extending $g \circ (u \ot \K[x])$ and \item since $g(1 \ot x) = \tilde{g}(y)$ and $g(y \ot 1)$ clearly commute in $B$, there exists a unique algebra morphism $\varphi: \K[y^{\pm 1}] \otimes \K[y^{\pm 1}] \to B$ such that $\varphi(y \ot 1) = g(y \ot 1)$ and $\varphi(1 \ot y) = \tilde{g}(y)$. \end{enumerate*} It follows that $\big(\K[y^{\pm 1}] \otimes \K[y^{\pm 1}], u'\ot \K[y^{\pm1}], \K[y^{\pm1}] \ot \pi\big)$ is the pushout of $(\pi, u \otimes \K[x])$ in $\algk$. On the other hand, to compute $A \bul (H \bul H)$ one observes that if $B$ is a $\K$-algebra with $\K$-algebra morphisms $f:\K[y^{\pm 1}] \to B$, $g:\K[y^{\pm1}] \otimes \K[x] \to B$ such that $f \circ \pi = g \circ (\pi \ot \K[x]) \circ \Delta$ then: \begin{enumerate*}[label=(\roman*)] \item since $g(y \ot 1)g(1 \ot x) = g(\pi(x) \ot x) = f(y) \in B^{\times}$, $g(1 \ot x)$ is invertible in $B$ too, and hence there exists a unique morphism of $\K$-algebras $\tilde{g}:\K[y^{\pm 1}] \to B, y \mapsto g(1 \ot x),$ and \item since $g(1 \ot x) = \tilde{g}(y)$ and $g(y \ot 1)$ clearly commute in $B$, there exists a unique algebra morphism $\varphi: \K[y^{\pm 1}] \otimes \K[y^{\pm 1}] \to B$ such that $\varphi(y \ot 1) = g(y \ot 1)$ and $\varphi(1 \ot y) = \tilde{g}(y)$. \end{enumerate*} This implies that $\big(\K[y^{\pm1}] \ot \K[y^{\pm1}],\Delta',\K[y^{\pm1}] \ot \pi\big)$ is the pushout of $\big(\pi, (\pi \ot \K[x]) \circ \Delta\big)$ in $\algk$. 
Therefore, $\theta = \id$ and coassociativity and counitality are given by $\Delta'(1) = 1 \ot 1$ and $\varepsilon'(1) = 1$. 
Summing up, $\left(\K,\K\left[y^{\pm 1}\right],\pi,u'\right)$ is a geometric partial $\K[x]$-comodule in $\algk$. 
However, 
\[
Y = \mathsf{Eq}\big(u \ot\K[x], (\pi\ot\K[x])\circ \Delta\big) = \left\{p(x)\in\K[x]~\left|~ 1\ot p(x) = \sum_{i}p_iy^i\ot x^i\right.\right\} = \K
\]
and clearly 
\[
\xymatrix @!0 @C=50pt @R=20pt{
 & \K \ar@{=}[dl] \ar[dr]^-{u} & \\
\K \ar[dr]_-{u} & & \K[x] \ar@{->>}[dl]^-{\pi} \\
 & \K\left[y^{\pm 1}\right] & 
}
\] 
cannot be a pushout diagram. Notice also that working instead with $\cC = \Calgk$, the category of commutative $\K$-algebras, the same argument leads to the same conclusion.
\end{example}

Since Theorem \ref{thm:globalization} shows that the globalization, whenever it exists, must be obtained as the equalizer of \eqref{eq:glob1}, we get an obstruction for its existence in the category of (commutative) algebras.

\begin{corollary}[of Theorem \ref{thm:globalization}]\label{obstruction}
In the categories $\algk$ and $\Calgk$ of (commutative) algebras over a field $\K$, a general globalization for geometric partial comodules does not exist.
\end{corollary}

\begin{proof}
This follows directly from Theorem~\ref{thm:globalization} in combination with Example~\ref{ex:NO}.
\end{proof}

A second important consequence of Theorem \ref{thm:globalization} is that the globalization construction is, in fact, functorial and it provides a right adjoint to the inclusion functor $\cI$.

\begin{corollary}[of Theorem \ref{thm:globalization}]\label{cor:adjunction}
Every global comodule is globalizable as geometric partial comodule. That is, the functor $\cI:\com{H}\to \gparcom{H}$ corestricts to a fully faithful functor
$$\cJ:\com{H}\to \gparcom{H}_{gl}.$$
Moreover, the assignment $X\mapsto Y_X$ induces a functor
\[
\cG:\gparcom{H}_{gl}\to \com{H}
\]
which is right adjoint to the fully faithful functor $\cJ:\com{H}\to\gparcom{H}_{gl}$. 
\end{corollary}

\begin{proof}
For any global comodule $Y$, the identity morphism $\id:Y\to Y$ satisfies all axioms \ref{item:GL0}--\ref{item:GL3} and hence $Y$ is the globalization of $\cI(Y)$. Thus, the image of the fully faithful functor $\cI:\com{H}\to \gparcom{H}$ lies in the full subcategory of globalizable partial comodules.

Moreover for any globalizable partial comodule $X$, axioms \ref{item:GL0} and \ref{item:GL3} entail that we have a (global) comodule $\cG(X) \coloneqq Y$ and a universal arrow $p:\cI(Y) \to X$ from $\cI$ to $X$, which by Theorem~\ref{thm:globalization} we can realize as the equalizer $(Y_X,\kappa)$ of \eqref{eq:glob1} and as the morphism $\epsilon_X \coloneqq (X \ot \varepsilon) \circ \kappa $, respectively. Therefore, by \cite[\S{IV.1}, Theorem 2(iv)]{MacLane}, $X\mapsto Y_X$ is the object function of a functor $\cG:\gparcom{H}_{gl}\to \com{H}$ which is right adjoint to $\cJ$.
\end{proof}

\begin{remarks}\label{rem:adjs}
\begin{enumerate}[label=(\alph*),leftmargin=0.7cm]
\item For any global $H$-comodule $(Y,\delta)$ it is well-known that the following diagram is an absolute equalizer
$$\xymatrix @C=40pt{Y \ar[r]^-\delta & Y\ot H \ar@<.5ex>[r]^-{\delta\ot H} \ar@<-.5ex>[r]_-{X\ot \Delta} & Y\ot H\ot H}.$$
As this equalizer is exactly the equalizer of \eqref{eq:glob1} applied the case $X=\cJ(Y)$, we find that $Y\cong \cG\cJ(Y)$ and this isomorphism describes the unit of the adjunction $(\cJ,\cG)$, which reconfirms that $\cJ$ is a fully faithful functor. On the other hand, for any globalizable partial comodule $X$ the counit $\epsilon_X:\cJ\cG(X)\to X$ of the adjunction is given exactly by 
$$\epsilon_X=(X\ot \varepsilon)\circ \kappa,$$
where $\kappa$ is the equalizer of \eqref{eq:glob1}, and which is a morphism of $\gparcom{H}$ by Lemma \ref{constructionpreglob}. 

\item The conclusion of Corollary \ref{cor:adjunction} 
can be rephrased by saying that $\com{H}$ is a coreflexive subcategory of $\gparcom{H}_{gl}$, because it is a full subcategory whose inclusion functor admits a right adjoint, and that $Y_X$ is the coreflector in $\com{H}$ of $X$.
\end{enumerate}
\end{remarks}

We conclude this subsection with a remarkable result, showing that the adjunction of Corollary \ref{cor:adjunction} provides a splitting of the classical free-forgetful adjunction for global comodules.
By carefully inspecting the proof of \cite[Proposition 2.20]{JoostJiawey}, one realizes that if $V \ot \varepsilon: V \ot H\to V$ is an epimorphism for every object $V$ in $\cC$, then the forgetful functor $U:\gparcom{H}\to \cC, (X,X\bul H,\pi_X,\rho_X)\mapsto X$, admits a right adjoint $T$ given by the so-called \emph{trivial partial comodule construction}. Namely, for every $V\in\cC$ one puts $V\bul H=V$, $\pi_V=V\otimes \varepsilon$ and $\rho_V = \id_V$. 
This makes $(V,V\bul H,\pi_V,\rho_V)$ a geometric partial $H$-comodule. 
The next proposition tells that the trivial partial comodule structure on an object $V$ in $\cC$ is always globalizable and that its globalization is exactly the usual free comodule over $V$, supporting the fact that geometric partial comodules and the trivial-forgetful adjunction are a refinement of usual comodules and the well-known free-forgetful adjunction.

\begin{proposition}\label{prop:adj}
Assume that $V \ot \varepsilon:V \ot H \to V$ is an epimorphism in $\cC$ for every object $V$.
With notations as above, the free-forgetful adjunction between $\cC$ and $\com{H}$ factors through the category of globalizable geometric partial comodules as in the following diagram, where the inner and outer triangles commute.
\[
\xymatrix @=17pt{
\cC 
\ar@<-.4ex>[dr]_-T \ar@<.4ex>[rr]^-{-\ot H} && {\com{H}} \ar@<-.4ex>[dl]_(0.55){\cJ} \ar@<.4ex>[ll]^{}\\
&\gparcom{H}_{gl} \ar@<-.4ex>[ur]_-{\cG} \ar@<-.4ex>[ul]_(0.45){U}  
}
\]
\end{proposition}

\begin{proof}
In view of Lemma \ref{constructionpreglob}, it is easy to check that $(V \ot H, V \ot \Delta)$ is the globalization of $(V,V,V \ot \varepsilon,\id_V)$. Thus, $T(V)$ is always globalizable.
Obviously, $U\circ \cJ:\com{H}\to \cC$ coincides with the forgetful functor and, as $\cG$ is right adjoint to $\cJ$ and $T$ is right adjoint to $U$, it follows by uniqueness of the right adjoint that $\cG\circ T\cong -\ot H:\cC\to \com{H}$.
\end{proof}

\subsection{Globalization versus global covers}\label{preglob}

By definition, a globalizable partial comodule is induced by a global comodule. Conversely, we can now finally show that any induced partial comodule is globalizable and moreover that there is an equivalence between globalizable partial comodules and minimal proper covers.

\begin{theorem}\label{globcov}
Let $H$ be a coalgebra in $\cC$ for which the equalizer of the pair \eqref{eq:glob1} exists in $\com{H}$ for any geometric partial comodule (e.g.~a coalgebra $H$ for which $\com{H}$ is complete). 

If $X$ is a geometric partial comodule that has been induced by a global comodule, then $X$ is globalizable. In other words the functor $\Ind$ from Proposition~\ref{pr:covers} co-restricts to a functor 
$$\Ind:\Cov^H\to \gparcom{H}_{gl},$$
which has a fully faithful right adjoint ${\sf Gl}$ given by ${\sf Gl}(X)=(\cG(X),X,\epsilon_X)$. Moreover, for any globalizable partial comodule, ${\sf Gl}(X)$ is a minimal proper cover in the sense of Definition \ref{def:minprop} and the functors $\Ind$ and ${\sf Gl}$ induce an equivalence of categories
\[\xymatrix{\Cov^H_{pr,min} \ar@<.5ex>[rr]^-{\Ind} & \sim & \gparcom{H}_{gl} \ar@<.5ex>[ll]^-{\sf Gl}}.\]
\end{theorem}

\begin{proof}
Let $(Y,X,p)$ be a cover and $\Ind(Y,X,p)=(X,X\bul H, \pi_X,\rho_X)$ be the induced partial comodule. By definition, $p:Y\to X$ is a morphism of geometric partial comodules. 
Then, by Lemma~\ref{constructionpreglob} there exists a unique $H$-colinear morphism $\tilde{p}:Y\to Y_X$ such that $\epsilon_X \circ \tilde{p} = p$, where $Y_X$ is the equalizer \eqref{eq:glob1} in $\com{H}$. 
Now consider the following diagram
\[
\xymatrix @!0 @C=60pt @R=23pt{
 & Y \ar[dddl]_-{p} \ar[dddr]^(0.55){\ (p\ot H)\circ \delta_Y} \ar[dd]^{\tilde p} & \\ \\
  & Y_X \ar[dl]_-{\epsilon_X} \ar[dr]^-{\kappa} & \\
X \ar[dr]_-{\rho_X} & & X\ot H \ar@{->>}[dl]^(0.45){\pi_X} \\
 & X\bul H \ar@{}[u]|<<<{\pushout} & 
}
\]
The outer square is a pushout square, because $X$ is induced by $Y$. Since the upper part of the diagram commutes, it follows easily that the inner square is also a pushout square and hence $Y_X$ is a globalization of $X$ by Theorem~\ref{thm:globalization}. This shows that the functor $\Ind$ indeed corestricts to the category of globalizable geometric partial comodules. 

The functor ${\sf Gl}$ is obviously well-defined and the adjunction property follows easily from the adjunction $(\cJ,\cG)$ in Corollary \ref{cor:adjunction}. 
As any globalizable geometric partial comodule is induced by its globalization, we have that $\left(\Ind\circ {\sf Gl}\right)(X)\cong X$, from which it follows that ${\sf Gl}$ is fully faithful. 

If $X$ is a globalizable geometric partial comodule, then we know from Theorem~\ref{thm:globalization} that its globalization $\cG(X)=Y_X$ is co-generated by $X$ as a global comodule, that is, that the cover ${\sf Gl}(X)$ is proper. Moreover, by the universal property \ref{item:GL3} of the globalization $Y_X$, the proper cover ${\sf Gl}(X)$ is also minimal. 

Finally, if $(Y,X,p)$ is a minimal proper cover, then we know from the above that ${\sf Gl}(X)=(Y_X,X,\epsilon_X)$ is also a proper cover and moreover $\epsilon_X \circ \tilde{p} = p$. Then the minimality implies that $\tilde{p}$ is an isomorphism and hence $(Y,X,p)$ and ${\sf Gl}(X)$ are isomorphic covers, from which we deduce that the restriction of the induction functor to minimal proper covers is also fully faithful and so we have the required equivalence of categories.
\end{proof}

\begin{remark}
Observe that the functor $\ms{Gl}: \gparcom{H}_{gl} \to \Cov^H_{pr,min}$ is always well-defined. The original contribution of Theorem \ref{globcov} is the fact that $\ms{Ind}: \Cov^H_{pr,min} \to \gparcom{H}_{gl}$ is well-defined, too.
\end{remark}

\subsection{Conclusions, examples and applications}

We showed that there exists a general procedure to compute the globalization of a geometric partial comodule, whenever this globalization exists (Theorem \ref{thm:globalization}). Our approach also provides an obstruction for its existence in certain categories, such as the category of (commutative) algebras. In forthcoming papers \cite{Saracco-Vercruysse1, Saracco-Vercruysse2} we will show that globalization exists in many cases of interest such as partial actions of monoids on sets, geometric partial coactions in abelian categories, partial comodule algebras and partial (co)representations of Hopf algebras. Several globalization theorems appearing in literature are hence subsumed as particular instances of our results.

To finish this paper, we provide some examples of explicit globalizations of (induced) geometric partial comodules.

\begin{example}\label{ex:global2}
Assume that we are in the situation of Example \ref{ex:global}, that is, that we have a surjective morphism of global comodules $p : Y \to X$. As we have seen, the geometric partial comodule structure induced on $X$ by $Y$ via $p$ is the global one $(X,\delta_X)$. In addition, being global, $(X,\delta_X)$ is already the absolute equalizer of $(\delta_X \ot H, X \ot \Delta)$ and so it follows that the globalization of the induced geometric partial comodule structure is still the starting global comodule structure. 
\end{example}

\begin{example}\label{ex:geometric}
Consider $G\coloneqq (\R,+,0)$ and $S\coloneqq \R$. Then the action $\beta:G\times S \to S, (g,s)\mapsto g+s,$ of $G$ on $S$ by translation can be seen as the action of an affine algebraic group on an affine set. Consider $V\coloneqq \{\pm 1\} = \cZ(X^2-1) \subseteq \R$. Then we can look at the restriction $\alpha$ of $\beta$ to $V$ as in Example \ref{ex:induced}. In this setting,
\[
G \bullet V = \left\{ (g,v) \mid v\in V_{g^{-1}}\right\} = \left\{ (0,\pm1),(2,-1),(-2,1)\right\} = \cZ(Z^2-1,X^2+2XZ) \subseteq \R^2
\]
is an affine set as well and the diagram
\[
\xymatrix @!0 @C=60pt @R=20pt{
V & & G\times V \\
 & G\bullet V \ar[ul]^{\alpha} \ar[ur]_{\subseteq} &
}
\]
is composed by polynomial maps, so that this provides an example of a ``geometric partial action''. Let us show that the globalization of this partial action gives back the whole line.

Passing to the ring of coordinates, we obtain a Hopf algebra $H\coloneqq \R[X]$ (with $X$ primitive) and a geometric partial $H$-comodule structure on the algebra $A\coloneqq \R[Z]/\langle Z^2-1\rangle \eqqcolon \R[z]$ which is given as follows. Set $\R[x,z] \coloneqq {\R[X,Z]}/{\langle Z^2-1,X^2+2XZ \rangle}$,
\[
\pi_A : \R[X] \otimes \R[z] \to \R[x,z], \ \ \begin{cases}X\otimes 1 \mapsto x \\ 1\otimes z \mapsto z \end{cases}
\quad
\text{and}
\qquad
\rho_A : \R[z] \mapsto \R[x,z], \ \ z \mapsto x+z.
\]
Then
\vspace{-5pt}
\[
\xymatrix @!0 @C=55pt @R=25pt{
A \ar[dr]_-{\rho_A} & & H\otimes A \ar@{->>}[dl]^-{\pi_A} \\
 & H\bullet A &
}
\]
is a geometric partial $H$-comodule structure on $A$ in the category of affine algebras. Observe that, since $(A,H\bul A,\pi_A,\rho_A)$ is an induced geometric partial comodule, 
the equalizer $(Y_A,\delta)$ of \eqref{eq:glob1} is the globalization of $A$, by Theorem \ref{thm:globalization} and Theorem \ref{globcov}. 
Consider then the equalizer $Y_A$ of the pair $(H \otimes \rho_A, (H \otimes \pi_A)(\Delta \otimes A))$. Since
\[
\begin{aligned}
& (H \otimes \rho_A)(X \otimes 1) = X\otimes 1, & \qquad  & \big((H \otimes \pi_A)(\Delta \otimes A)\big)(X \otimes 1) = X\otimes 1 + 1 \otimes x, \\ 
& (H \otimes \rho_A)(1\otimes z) = 1\otimes x + 1\otimes z, & \qquad  & \big((H \otimes \pi_A)(\Delta \otimes A)\big)(1\otimes z) =  1\otimes z,
\end{aligned}
\]
it follows that $X\otimes 1 + 1 \otimes z \in Y_A$ and we have a well-defined algebra map $\psi:\R[X] \to Y_A$, $X \mapsto X\otimes 1 + 1 \otimes z$. It can be shown, with a bit of effort, that $\psi$ is an isomorphism.
\end{example}

\begin{example}\label{ex:geometric2}
Analogously to Example \ref{ex:geometric}, consider $G=\sf{SO}(2,\R)$ acting on $\R^2$ and $V= \{a\coloneqq (1,0)\} = \cZ(X-1,Z)$. In this setting, $G\bullet V = \{(I_2,a)\}$ together with the inclusion $G\bullet V\subseteq G\times V$ and the map $G\bullet V \to V, (I_2,a)\mapsto a,$ gives a partial action of $G$ on $V$. By passing to the coordinate rings we find a geometric partial $H$-comodule structure on $\R$, where $H=\R[\sf{SO}(2,\R)]$. Namely, the trivial geometric partial comodule structure
\[
\xymatrix @!0 @C=35pt @R=18pt{
\R \ar@{=}[dr] & & H \ar@{->>}[dl]^-{\varepsilon} \\
 & \R &
}
\]
By Proposition \ref{prop:adj}, the equalizer $Y_\R$ is $H$, which corresponds to the unit circle in $\R^2$.
\end{example}



\end{document}